\documentclass{article}

\usepackage{amscd,amssymb,verbatim,graphicx,stmaryrd,amsthm,color} 
\usepackage{hyperref}
\usepackage{diagrams}
\usepackage[margin=1.5in]{geometry}

\input xy    %Loads xy-Pic
\xyoption{curve}  %Inputs curve things (See section 8 of manual)
\xyoption{arc} %Fancy version of the previous
\xyoption{frame} %Frame extension. Needed for frm command
\xyoption{tips} %For fancy arrows
\xyoption{arrow} %Needed for \ar command 
\xyoption{color} %Color

\newtheorem{theorem}{Theorem}[section]  %This keeps track of the section in the naming of the theorems
\newtheorem{lemma}[theorem]{Lemma}  %This creates a lemma environment and numbers it along with the theorems
\newtheorem{proposition}[theorem]{Proposition}

\newtheorem{remark}[theorem]{Remark}
\newtheorem{example}[theorem]{Example}
\newcounter{theexercise} 

{\begin{list}{}
         {\setlength{\leftmargin}{#1}}
         \item[]}
{\end{list}}

\newcommand{\End}{\mathrm{End}}

\newcommand{\intd}[1]{\displaystyle \int_{#1}}
\newcommand{\intdd}[2]{\displaystyle \int_{#1}^{#2}}
\newcommand{\fracd}[2]{\displaystyle \frac{#1}{#2}}
\newcommand{\bb}[1]{\mathbb{#1}}

\newcommand{\defeq}{\mathrel{\mathpalette{\vcenter{\hbox{$:$}}}=}}
\newcommand{\SU}{\mathrm{SU}}
\newcommand{\SO}{\mathrm{SO}}
\newcommand{\U}{\mathrm{U}}
\newcommand{\GL}{\mathrm{GL}}

\newcommand{\PU}{\mathrm{PU}}

\newcommand{\A}{{\mathcal{A}}}
\newcommand{\G}{{\mathcal{G}}}

\newcommand{\PSU}{\mathrm{PSU}}

\newcommand{\afv}{v}

\newcommand{\afA}{A}

\newcommand{\afa}{a}

\newcommand{\afu}{u}
\title{On the components of the gauge group \\ for $\PU(r)$-bundles}

\author{David L. Duncan}

\date{}

\begin{document}

\maketitle

\begin{abstract}
We discuss a general procedure for using characteristic classes to study the components of the gauge group for a principal $G$-bundle. To illustrate this, we work out the case where $G$ is the projective unitary group. 
\end{abstract}

\tableofcontents

\section{Introduction}

Fix a Lie group $G$ and a space $X$. The objective of these notes is to describe a simple strategy for using the characteristic classes of $G$-bundles over $X$ to glean information about the components of the gauge group of any fixed $G$-bundle. Rather than stating any general theorems to this effect, we illustrate the techniques by working out the specific case where $G = \PU(r)$ is the projective unitary group, and $X$ is smooth manifold with low dimension. We note that these are by no means necessary restrictions, and the techniques we describe here can often be used in much more general settings (e.g., when $X$ is a CW complex of higher dimension, or $G$ is a some other Lie group). For concreteness, we work in the smooth category, so all spaces and maps are assumed to be smooth. Those interested in working in other categories (e.g., topolgical, CW) can, for the most part, simply reinterpret the words `space', `map', etc. 

As an application, we arrive at several results regarding the degree and parity of gauge transformations on $\PU(r)$-bundles. These are all standard for the case $r = 2$, but our proofs differ in flavor from many of those appearing elsewhere in the literature, e.g., \cite{Flinst} \cite{DS2}. For example, we provide an alternate proof of a result of Dostoglou and Salamon \cite{DS2} that a non-trivial bundle on a closed 3-manifold has a degree 1 gauge transformation (see Lemma \ref{degprop} for a precise statement).

		\section{Associated bundles and the gauge group}\label{AssocaiatedBundlesAndTheGaugeGroup}
		
		Let $G$ be a Lie group, $X$ a manifold and $P \rightarrow X$ a principal $G$-bundle (our convention is that the action of $G$ on $P$ is a \emph{right} action). We will say that two principal $G$-bundles $P \rightarrow X$ and $P' \rightarrow X$ are {\bfseries isomorphic (as $G$-bundles)} if there is a $G$-equivariant bundle map $\phi: P \rightarrow P'$ covering the identity. We will refer to such a $\phi$ as a {\bfseries $G$-bundle isomorphism}.

		Suppose we are given a right action $\rho: G \rightarrow \mathrm{Aut}(F)$ of $G$ on a manifold $F$. Then we can form the {\bfseries associated bundle}

$$P \times_G F \defeq (P \times F) / G$$
where $G$ acts diagonally on $P \times F$. The space $P \times_G F$ is naturally a fiber bundle over $X$ with fiber $F$. If $F$ has additional structure, and the action $\rho$ respects this structure, then the fibers of $P \times_G F \rightarrow X$ inherit this additional structure. 

\begin{example}\label{ex0}
Suppose $G$ acts on a vector space $V$ by linear transformations. Then 

$$P(V) \defeq P \times_G V$$ 
is naturally a \emph{vector} bundle over $X$. 
\end{example}

\begin{example}\label{ex1}
Suppose $H$ is a second Lie group and $G$ acts on $H$. Assume in addition that the action of $G$ on $H$ commutes with the action of $H$ on itself given by right multiplication. Then the space $P \times_G H$ is naturally equipped with the structure of a principal $H$-bundle over $X$. For example, take $G = \SU(r)$ and $H = \PU(r)$ (see the next section for a review of $\PU(r)$). Then $\SU(r)$ acts on $\PU(r)$ by left multiplication of the inverse via the homomorphism $\SU(r) \rightarrow \PU(r)$. This action commutes with right multiplication, so 

$$P \times_{\SU(r)} \PU(r) \longrightarrow X$$
is a principal $\PU(r)$-bundle.
\end{example}

\begin{example}\label{ex2}
Consider action of $G$ on itself by conjugation. This allows us to form the bundle

$$P \times_G G \longrightarrow X.$$
Note, however, that though the fibers of $P \times_G G$ are diffeomorphic to the group $G$, this is not a principal $G$-bundle in a canonical way. This is because the multiplication action of $G$ on itself does not, in general, commute with conjugation.
\end{example}

		A {\bfseries gauge transformation} on $P$ is a $G$-bundle isomorphism from $P$ to itself. The set of gauge transformations on $P$ forms a group, called the {\bf gauge group}, and is denoted ${\G}(P)$. On may equivalently view the gauge group as the set $\mathrm{Map}(P, G)^G$ of $G$-equivariant maps $P \rightarrow G$. Here $G$ acts on itself by the (right action of) conjugation. That is, the right action of $g \in G$ on $G$ is the map $G \rightarrow G$ given by $h \longmapsto g^{-1} h g$. Then the identification
	
				$$\G(P) = 	\mathrm{Map}(P, G)^G$$	
		follows by sending ${\afu}: P \rightarrow P$ to the map $g_{\afu}: P \rightarrow G$ defined by the formula
			
			$${\afu}(p) = p \cdot g_{\afu}(p).$$
		In general, we will typically not distinguish between ${\afu}$ and $g_{\afu}$. 
		
		There is a third equivalent formulation of the gauge group. In this formulation one views gauge transformations as sections of the bundle $P \times_G G \rightarrow X$ from Example \ref{ex2}. It is an easy exercise to show there is a canonical identification $\mathrm{Map}(P, G)^G = \Gamma(P \times_G G)$.
		
		Viewing the the gauge group as a space of maps, it is equipped with a natural topology.\footnote{Actually, there are many natural topologies one could put on $\G(P)$ (e.g., the topology of uniform convergence or the topology of uniform convergence in all (or some) derivatives). However, the set of connected components will be the same, and so we can unambiguously write $\pi_0 \: \G(P)$.} In these notes we are interested in the set of connected components $\pi_0 \: \G(P)$. This notation makes sense because the gauge group is locally path-connected, and so the connected components are exactly the path-connected components. We will use
		
		$$\G_0(P) \subseteq \G(P)$$ 
		to denote the component containing the identity.

		\section{The general strategy}\label{TheGeneralStrategy}
		
					Here we sketch the general strategy for using characteristic classes to study the components of the gauge group. The crux of the matter is the following observation by Donaldson \cite{Donfloer}. We include a proof here for convenience.
		
		\begin{proposition}{\cite[Section 2.5.2]{Donfloer}}\label{donfloer}
		Suppose $G$ is a compact Lie group, $X$ is a smooth manifold and $P \rightarrow X$ is a principal $G$-bundle. Then there is a bijection between $\pi_0(\G(P))$ and the set of isomorphism classes of principal $G$-bundles over $S^1 \times X$ that restrict to $P$ on a fiber. This bijection is induced from the map that sends a gauge transformation $\afu$ to the bundle
		
		\begin{equation}
		P_{\afu} \defeq \left[0, 1 \right] \times P / (0, \afu(p)) \sim (1, p),
		\label{bundleu}
		\end{equation}			
		given by the mapping torus of $\afu$. 
		\end{proposition}
		
		\begin{proof}
			Well-defined: The isomorphism class of the bundle $P_{\afu}$ depends only on the path component of ${\afu}$ in $\G(P)$. The gauge group is locally path-connected (it is locally modeled on the vector space consisting of sections of the bundle $P(\frak{g})$ defined by the Adjoint representation), so the connected components are the path components, and the map (\ref{bundleu}) descends to a give a well-defined map from $\pi_0(\G(P))$ to the isomorphism classes of principal $G$-bundles over $S^1 \times X$ that restrict to $P$ on a fiber. 
				
				Surjectivity: Suppose we are given any bundle $Q \rightarrow S^1 \times X$ with $Q\vert_{\left\{1\right\} \times X} = P$, and consider the obvious projection $\pi : Q \rightarrow S^1$. Since $G$ is compact, $Q$ admits a $G$-invariant metric. (This can be obtained by first choosing any metric $\left( \cdot, \cdot \right)$ and then declaring 
				
				$$(v, w)_{\mathrm{inv}} \defeq \frac{1}{\mathrm{vol}(G)} \intd{G} ( DR_g v, DR_g w)\: d\mathrm{vol}_G,$$
				 where $DR_g$ is the pushforward of multiplication by $g \in G$ and we are using an invariant Haar measure to define the integral on $G$.) Let $\Phi_t: Q \rightarrow Q$ denote the time-$t$ gradient flow of $\pi$, normalized so $\Phi_{1}$ maps each fiber to itself (this is just saying the circle has length 1). The $G$-invariance implies that $\Phi$ is $G$-equivariant. Then ${\afu} \defeq \Phi_{1}\vert_{\pi^{-1}(1)} : P \rightarrow P$ is the desired gauge transformation.
				
				Injectivity: Suppose there is some ${\afu}\in \G(P)$ with $\Psi: P_{\afu} \stackrel{\cong}{\rightarrow} S^1 \times P$. Let $\Phi_\bullet: I \times P_{\afu} \rightarrow P_{\afu}$ be the gradient flow as constructed in the previous paragraph, and $\pi: S^1 \times P \rightarrow P$ the projection. Then consider the composition
				
				$$\pi \circ \Psi \circ \Phi_\bullet \circ \Psi^{-1}\vert: I \times P  \longrightarrow P.$$
				where, in the domain, we have set $P = \left\{1 \right\} \times P \subset Q$. Since everything is equivariant, this is a path in $\G(P)$ from ${\afu}$ to the identity map.
		\end{proof}

		Proposition \ref{donfloer} serves as the vehicle for passing from topological information about principal bundles to topological information about gauge transformations. Topological information about principal $G$-bundles is captured by characterisic classes. For our purposes, we will say that a {\bfseries characteristic class for $G$} is a mapping $k$ that assigns to each principal $G$-bundle $P$ a cohomology class $k(P) \in H^*(X, R)$ for some ring $R$. We assume that $k$ is {\bfseries functorial} in the sense that for each map $f: Y \rightarrow X$, one has
		
		$$k(f^*P) = f^*k(P).$$
		It follows that $k(P)$ only depends on the $G$-bundle isomorphism class of $P$. The primary examples are the Chern classes for $\U(r)$- and $\SU(r)$-bundles, and the Pontryagin and Stiefel-Whitney classes for $G = \SO(r)$. In Section \ref{ClassificationOfPSU(r)Bundles} we will discuss certain characteristic classes for $G = \PU(r)$.

		Fix a gauge transformation ${\afu} \in \G(P)$, and let $P_{\afu}$ be the bundle from Proposition \ref{donfloer}. Suppose one has a preferred characteristic class $k$ for $G$. Then $k(P_{\afu})$ is unchanged under small perturbations of ${\afu}$, and so depends only on the connected component of ${\afu}$ in $\G(P)$. 
		
		Now suppose one has enough characteristic classes to classify all $G$-bundles over $S^1 \times X$. Then it follows immediately from Proposition \ref{donfloer} that two gauge transformations ${\afu}, {\afu'}\in \G(P)$ are in the same connected component if and only if these characteristic classes all agree on the bundles $P_{\afu}$ and $P_{\afu'}$. In the next section we carry this out explicitly for the case $G = \PU(r)$ and show how one can use this information to extract various existence results.

		\medskip
		
		Finally, we make a few remarks about the higher homotopy groups of the gauge group. We begin with $\pi_1\left(\G(P)\right)$. Given a loop $u : S^1 \rightarrow \G(P)$, one can form a bundle $R_{u} \rightarrow S^1 \times S^1 \times X$ by declaring $R$ to be the quotient
		
		$$\left[0, 1\right] \times S^1 \times P / (0, s, u(s)p) \sim (1, s, p).$$
		Then there is an analogue of Proposition \ref{donfloer} that says two loops $u, u'$ are homotopic if and only if the bundles $R_{u}$ and $R_{u'}$ are isomorphic. Then one can repeat the process above, and therefore use characteristic classes on $(\dim(X) + 2)$-dimensional manifolds to distinguish elements of $\pi_1(\G(P))$. Taking this generalization one step further, one can use characteristic classes to study $\pi_n(\G(P))$ for $n \in \bb{N}$. This can be a powerful strategy for small $n$, but the trade-off is that, for a given group $G$, it is not always the case that characteristic classes always detect non-isomorphic bundles over large-dimensional manifolds. For example, it is well-known that there are manifolds $Z$ and non-isomorphic $\U(r)$-bundles over $Z$ that have identical characteristic classes. On the other hand, we will see in the next section that if the dimension of $Z$ is no greater than $4$, then the characteristic classes for $\PU(r)$ distinguish $\PU(r)$-bundles over $Z$.

		\section{The case $G = \PU(r)$}	
			
		\subsection{Properties of $\PU(r)$}\label{PropertiesOfPSU(r)}
			
			Let $r \geq 2$ and consider the unitary group $\U(r)$. This group has center consisting of matrices of the form $\gamma \mathrm{Id}$ where $\gamma \in S^1 = U(1)$. We identify this center with $U(1)$, and the {\bfseries projective unitary group} is the quotient
			
			$$\PU(r) \defeq \U(r)/ \U(1).$$
			
			This group admits an equivalent description that is also quite useful. Consider the special unitary group $\SU(r) \subseteq \U(r)$. This has center $\bb{Z}_r$ consisting of matrices of the form $e^{2 \pi i d/r} \mathrm{Id}$ for $d \in \left\{0, \ldots, r-1\right\}$. Quotienting recovers the projective unitary group
			
			$$\PU(r) = \SU(r) / \bb{Z}_r.$$
			For this reason, $\PU(r)$ is often called the {\bfseries projective special unitary group} and is sometimes denoted $\PSU(r)$. 
			
			Set $G \defeq \PU(r)$ and let $\frak{g} \defeq  \frak{pu}(r)$ denote its Lie algebra. It follows from the previous paragraph that $G$ has trivial center, and is connected and compact. Furthermore, 
			
			$$\pi_1(G) \cong \bb{Z}_r,$$ 
			since $\SU(r)$ is simply-connected. Moreover, being the quotient of $\SU(r)$ by a discrete set, we have a Lie algebra isomorphism 
			
			$$\frak{g} \cong \frak{su}(r) = \left\{ \mu \in \mathrm{End}(\bb{C}^r) \: \vert \: \mu^* = - \mu \right\}.$$ 
			Hence $\frak{g}$ is simple and so it admits an $\mathrm{Ad}$-invariant inner product; moreover, this is unique up to multiplication by a positive scalar. Indeed, all such inner products are all of the form
			
			$$\langle \mu , \nu \rangle = \kappa_{r} \mathrm{tr}(\mu  \cdot \nu^*) = -\kappa_{r} \mathrm{tr}(\mu  \cdot \nu),$$
			for some $\kappa_{r} > 0$, where the trace is the one induced from the identification $\frak{g} \cong \frak{su}(r) \subset \mathrm{End}(\bb{C}^r)$, and $\mu \cdot \nu$ denotes matrix multiplication in $\mathrm{End}(\bb{C}^r)$. Here we describe several common normalizations for this inner product.
			
			\begin{itemize}
				\item Taking $\kappa_r = 1$ one obtains the {\bfseries Frobenius} inner product. This is the restriction to $\frak{g}$ of the standard inner product on the euclidean vector space $\mathrm{End}(\bb{C}^r) = \bb{R}^{2r^2}$.

				\item When $\kappa_{r} = 2r$ this is (the negative of) the {\bfseries Killing form}
				
				$$B: \frak{g} \otimes \frak{g} \longrightarrow \bb{R}$$
				In general, the Killing form can be abstractly defined by declaring that $B(\mu, \nu)$ is the trace of the operator 
				
				$$\mathrm{ad}(\mu) \circ \mathrm{ad}(\nu) = \left[\mu, \left[\nu, \cdot \right] \right].$$
				Here the trace is taken in the vector space
				
				$$\mathrm{End}(\frak{g}) \subset \mathrm{End} \left( \mathrm{End}(\bb{C}^r) \right) = \mathrm{End}\left(\bb{C}^{r^2} \right).$$
				 We will denote this by $\mathrm{Tr}$ to distinguish it from the trace $\mathrm{tr}$ on $\mathrm{End}(\bb{C}^r)$ above. Then a computation shows 
				
				\begin{equation}\label{traces}
				\mathrm{Tr}\left( \mathrm{ad}(\mu) \circ \mathrm{ad}(\nu) \right) = 2r\mathrm{tr}\left( \mu \cdot \nu^* \right).
				\end{equation}

				\item By taking $\kappa_{r} = 1/ 4\pi^2$ one obtains the inner product defined by declaring that the highest coroot of the $\frak{g}$ to have norm $\sqrt{2}$. That is, we consider the inner product such that if $\xi \in \frak{g} = \frak{su}(r)$ is any Lie algebra element mapping to $-\mathrm{Id} \in \SU(r)$ under the Lie-theoretic exponential map, then $\xi$ has norm $\sqrt{2}$. In terms of the standard basis coming from the embedding $\frak{g} \subset \End(\bb{C}^r)$, the element 

$$\xi = \left( \begin{array}{cccccc}
				0 & 2 \pi i  & 0 & \ldots & 0\\
				2 \pi i & 0  & 0& \ldots & 0\\
				0 & 0 & 0 & \ldots & 0\\
				\vdots & \vdots&\vdots && \vdots\\
				0 & 0 & 0 & \ldots & 0\\
				\end{array}\right)$$
				maps to $- \mathrm{Id} \in \SU(r)$, since it can be diagonalized to 
				
				$$\left( \begin{array}{cccccc}
				2\pi i & 0  & 0 & \ldots & 0\\
				0 & -2 \pi i  & 0& \ldots & 0\\
				0 & 0 & 0 & \ldots & 0\\
				\vdots & \vdots&\vdots && \vdots\\
				0 & 0 & 0 & \ldots & 0\\
				\end{array}\right)$$				
				Then $\xi$ has norm $\sqrt{2}$ when one uses the inner product $\langle \cdot, \cdot \rangle$ obtained by taking $\kappa_r = 1/4\pi^2$. 
						\end{itemize}

			 \noindent For our purposes, we leave the constant $\kappa_{r} > 0$ arbitrary, but fixed. In this way we hope that the reader will have an easier time matching up our formulas with those appearing elsewhere in the literature. 
			 
			 Having fixed an inner product on $\frak{g}$, we note that the adjoint map can be viewed as a representation of the form $\mathrm{Ad}: G \rightarrow \SO(\frak{g})$. Moreover, this representation is faithful. 
			
			Lastly, consider the action of $\U(r)$ on itself by conjugation. The center $\U(1)$ fixes every point in $\U(r)$, and so this action descends to an action of $G$ on $\U(r)$. We note that this $G$ action fixes the subgroup $\SU(r) \subset \U(r)$.

		\subsection{Classification of $\PU(r)$-bundles}\label{ClassificationOfPSU(r)Bundles}

		In \cite{Woody2}, L.M. Woodward exploited the adjoint representation to classify the principal $\PU(r)$-bundles over spaces of dimension up to $4$. This classification scheme assigns cohomology classes

		$$t_2(P) \in H^2(X, \bb{Z}_r), \indent q_4(P) \in H^4(X, \bb{Z})$$
		to each principal $\PU(r)$-bundle $P \rightarrow X$. For example, $q_4$ is defined to be the second Chern class of the complexified adjoint bundle $P(\frak{g})_{\bb{C}} \defeq P(\frak{g}) \otimes \bb{C}$,
		
				\begin{equation}
				q_4(P) \defeq  c_2\left( P(\frak{g})_{\bb{C}}\right) ,
				\label{defofchernq}
				\end{equation}
				where $\frak{g} = \frak{pu}(r)$ and $c_2$ is the second Chern class. The class $t_2$ is defined as the mod $r$ reduction of a suitable first Chern class.\footnote{The class $t_2$ is \emph{not} given by the first Chern class of $P(\frak{g})_{\bb{C}}$ (this is zero, see Example \ref{ex4}).} We will be mostly interested in the case where $X$ is a smooth manifold, but these classes are defined for CW complexes as well. 
				
\begin{example}
When $r = 2$ we have $\PSU(2) = \SO(3)$, and the classes $t_2$ and $q_4$ are exactly the second Stiefel-Whitney class and the (negative of the) first Pontryagin class, respectively. For example, the latter assertion follows because $c_2\left( P(\frak{g})_{\bb{C}}\right)  = - p_1\left( P(\frak{g})\right) $, and a rank 3 vector bundle with metric is isomorphic to its adjoint bundle (bundle of skew-symmetric endomorphisms) via the adjoint representation.
\end{example}

				\noindent Set $G  = \PU(r)$. We summarize the properties of these classes that we will need.

		\begin{itemize}
			\item If $\dim(X) \leq 4$ and $X$ is a manifold, then two bundles $P$ and $P'$ over $X$ are isomorphic if and only if $t_2(P) = t_2(P')$ and $q_4(P)= q_4(P')$;
			
		\medskip			
			
			\item The class $t_2(P)$ is zero if and only if the structure group of $P$ can be lifted to $\SU(r)$; that is, $P = \overline{P} \times_{\SU(r)} G$ for some principal $\SU(r)$-bundle $\overline{P} \rightarrow X$ \cite[p. 517]{Woody2}. Here the action of $\SU(r)$ on $G$ is by left multiplication via the projection $\SU(r)\rightarrow G$. 

		\medskip			
		
		\item The class $t_2(P)$ is the mod $r$ reduction of an integral class if and only if the structure group of $P$ can be lifted to $\U(r)$; that is, $P = \overline{P} \times_{\U(r)} G$ for some principal $\U(r)$-bundle $\overline{P} \rightarrow X$ \cite[p. 517]{Woody2}. Here the action of $\U(r)$ on $G$ is by left multiplication via the projection $\U(r)\rightarrow G$. 
		
		\medskip
			
			\item In the case of manifolds of dimension 2 or 3, the class $q_4$ is always zero, and $t_2$ determines a bijection between isomorphism classes of $G$-bundles and the space $H^2(X, \bb{Z}_r)$. In particular, if $X$ is a closed, connected, oriented surface or a connected, oriented elementary cobordism between two such surfaces, then $t_2$ is a bijection 
			
			$$t_2: \left\{ \begin{array}{c}
						\textrm{isomorphism classes of}\\
						\textrm{$G$-bundles over $X$}
						\end{array}\right\} \stackrel{\cong}{\longrightarrow}   \bb{Z}_r$$
						
						\medskip
						
			\item The mod $2r$ reduction of $q_4$ is a non-zero multiple of $C t_2$, where $C:H^2(X, \bb{Z}_r) \rightarrow H^4(X, \bb{Z}_{2r})$ is the Pontryagin square: 
			
			\begin{equation}
				q_4(P) = \left\{ \begin{array}{rc}
													(r+1)Ct_2(P)& \textrm{$r$ even}\\
													\fracd{r+1}{2} C t_2(P) & \textrm{$r$ odd}
													\end{array}\right.  \indent~~~~~~~~~ \mathrm{mod}\: 2r
			\label{pontrelation}
			\end{equation}
			In fact, when $X$ has dimension $\leq 4$, L.M. Woodward shows that the equivalence classes of $\PU(r)$-bundles correspond exactly to the pairs $(q, t) \in H^4(X, \bb{Z}) \times H^2(X, \bb{Z}_r)$ satisfying (\ref{pontrelation}). We also note that if $t \in H^2(X, \bb{Z}_{r})$ is the mod $r$ reduction of an integral class $\widetilde{t} \in H^2(X)$ then 
			
			\begin{equation}\label{computeC}
			Ct = \left\{ \begin{array}{rc}
										\widetilde{t} \smallsmile \widetilde{t} &\textrm{$r$ even}\\
												2\widetilde{t} \smallsmile \widetilde{t} 	 & \textrm{$r$ odd}
													\end{array}\right.  \indent~~~~~~~~~ \mathrm{mod}\: 2r.
							\end{equation}
			
			\medskip
			
			\item The classes $t_2$ and $q_4$ are functorial.
		\end{itemize}

			\subsection{Components of the gauge group}	\label{ComponentsOfTheGaugeGroup}

		Next, we use the characteristic classes $t_2$ and $q_4$ to study the components of the gauge group $\G(P)$ for a principal $\PU(r)$-bundle $P \rightarrow X$. Assume $X$ is a manifold of dimension at most $3$ (many of the definitions we give below have obvious extensions to manifolds of higher dimension). Set $G = \PU(r)$ and $\frak{g} = \frak{pu}(r)$. 
		
		Given ${\afu} \in \G(P)$, define $P_{\afu} \rightarrow S^1 \times X$ as in Section \ref{TheGeneralStrategy}, and consider the classes
		
		$$t_2(P_{\afu}) \in H^2(S^1 \times X, \bb{Z}_r), \indent q_4(P_{\afu}) \in H^4(S^1 \times X, \bb{Z}).$$		
		By the K\"{u}nneth formula, we have an isomorphism
		
		$$H^k(S^1 \times X, R) \cong H^k(X, R) \oplus H^{k-1}(X, R),$$%%%%Hatcher says this holds for cohomology when one space has freely generated R-cohomology
		where $R = \bb{Z}$ or $\bb{Z}_r$. The image of $t_2(P_{\afu}) \in H^2(X, \bb{Z}_r) \oplus H^1(X, \bb{Z}_r)$ in the first factor is exactly $t_2(P)$, so the dependence of $t_2(P_{\afu})$ on the gauge transformation ${\afu}$ is contained entirely in the projection of $t_2(P_{\afu})$ to the second factor $H^1(X, \bb{Z}_r)$. We denote this projection by 
		
		$$\eta({\afu}) \in H^1(X, \bb{Z}_r),$$
		and call this the {\bfseries parity} of ${\afu}$. We therefore have $t_2(P_{\afu}) = (t_2(P), \eta(u))$, with respect to the K\"{u}nneth splitting. 
		
		In a similar fashion, we will define the \emph{degree} of a gauge transformation. Before defining this explicitly, we first note that the K\"{u}nneth formula gives $H^4(S^1 \times X, \bb{Z}) = H^4(X, \bb{Z}) \oplus H^3(X, \bb{Z})$. Second, we have the following observation:
		
		\medskip
		
		\noindent \emph{Claim 1: The image of $q_4(P_{\afu})$ in $H^3(X, \bb{Z})$ is always even.} 

		\medskip		
		
		Assuming the claim for now, we define the {\bfseries degree} of ${\afu}$ to be the class $\deg({\afu}) \in H^3(X, \bb{Z})$ for which $2\deg({\afu}) $ is the image of $q_4(P_{\afu})$ in $H^3(X, \bb{Z})$. (This definition is valid regardless of the dimension of $X$.) 
		
		\begin{remark}
			Our orientation convention on $S^1 \times X$ is that $ds \wedge d\mathrm{vol}_X$ is a positive form when $d\mathrm{vol}_X$ is a positive form on $X$. Note that the opposite convention would yield a definition of the degree that is the negative of the one defined above. For example, the degree operator appearing in \cite{DS2} is the negative of the one here.
		\end{remark}

The proof of Claim 1 is just a characteristic class computation: We have $q_4(P_{\afu}) = c_2(P_{\afu}(\frak{g}) \otimes \bb{C})$. The mod 2 reduction of the total Chern class of a complex vector bundle is the total Stiefel-Whitney class of its underlying real vector bundle (see \cite[p. 171]{MilS}). The underlying real vector bundle of $P_{\afu}(\frak{g}) \otimes \bb{C}$ is $P_{\afu}(\frak{g})  \oplus P_{\afu}(\frak{g})$. So working mod 2, we have

$$q_4(P_{\afu})  \equiv_2  w_4(P_{\afu}(\frak{g})  \oplus P_{\afu}(\frak{g}) )  \equiv_2  w_2(P_{\afu}(\frak{g}))^2$$
By the K\"{u}nneth formula again, we can write $w_2(P_{\afu}(\frak{g})) = a + b \smallsmile ds$, where $a \in H^2(X, \bb{Z}_2)$, $b \in H^1(X, \bb{Z}_2)$ and $ds$ is the generator of $H^1(S^1,\bb{Z}_2)$. So we have

$$q_4(P_{\afu}) \equiv_2 a^2 + 2 a \smallsmile b \smallsmile ds \equiv_2 a^2.$$
However, the term $a^2$ is a 4-form on $X$ and this vanishes when we project to $H^3(X, \bb{Z})$. This proves Claim 1.

\medskip

			Proposition \ref{donfloer}, and the properties of $t_2$ and $q_4$ immediately imply that the parity and degree detect the components of the gauge group in low dimensions:

\medskip		
		
		\begin{itemize}
		\item If $\dim X \leq  2$, or if $\dim X = 3$ but $X$ is \emph{not} closed, then there is an injection 
		
			$$\begin{array}{rcl}
			\eta: \pi_0(\G(P)) &\hookrightarrow & H^1(X, \bb{Z}_r).\\
			\end{array}$$
			In particular, a gauge transformation $\afu$ lies in the identity component $\G_0(P)$ if and only if $\eta(\afu) = 0$.

		\medskip
		
		\item If $\dim X = 3$ and $X$ is closed, connected and oriented, then there is an injection 
		
		$$(\eta, \mathrm{deg}): \pi_0(\G(P))	\hookrightarrow H^1(X, \bb{Z}_r) \times \bb{Z}.$$
		In particular, a gauge transformation $\afu$ lies in the identity component $\G_0(P)$ if and only if $\eta(\afu) = 0$ and $\deg(\afu) = 0$. 
		
		\end{itemize}
		
		\medskip
		
		\noindent  See \cite{DS2}, \cite{Flinst} for alternative realizations of the parity and degree.

		The next proposition describes some of the data captured by the parity.

		\begin{proposition}\label{trivgaugetrans}
				Fix a gauge transformation ${\afu}: P \rightarrow \PU(r)$. Then the following are equivalent:

				\begin{itemize}
					\item $\eta({\afu}) = 0$;
					\item ${\afu}: P \rightarrow \PU(r)$ lifts to a $\PU(r)$-equivariant map $\widetilde{{\afu}}: P \rightarrow \SU(r)$;
					\item When restricted to the 1-skeleton of $X$, ${\afu}$ is homotopic to the identity map.
				\end{itemize}
				
				Moreover, if $X$ is a compact, connected, oriented 3-manifold and $\eta({\afu}) = 0$, then $\deg({\afu})$ is divisible by $r$.
		\end{proposition}

		\begin{proof} 
			Set $G  = \PU(r)$. First note that the final assertion of the proposition follows from (\ref{pontrelation}). To prove the equivalence of the three conditions, we identify an equivalent characterization of the parity (this is the definition given in (\cite{DS2})). Any gauge transformation ${\afu}: P \rightarrow G$ determines a group homomorphism
			
			$${\afu}_*: \pi_1(P) \longrightarrow \pi_1(G) = \bb{Z}_r$$
			in the usual way. 
			
			\medskip
			
			\noindent \emph{Claim 2: This descends to a homomorphism $\pi_1(X) \rightarrow \bb{Z}_r$.}			
			
			\medskip
			
			By the homotopy long exact sequence, it suffices to show that if $\left[\gamma\right]$ lies in the image of $\pi_1(G) \rightarrow \pi_1(P)$ then ${\afu}_* \left[\gamma \right] = 0$. The condition on $\gamma$ implies that this class is represented by a loop of the form $t \mapsto p \cdot \mu(t)$, where $p \in P$ is fixed and $\mu:  \bb{R} / \bb{Z} \rightarrow G$ is a loop. Then the gauge invariance of ${\afu}$ implies that ${\afu}_*\left[ \gamma \right]$ is the homotopy class of the loop $t \mapsto \mu(t)^{-1} {\afu}(p) \mu(t)$. The key point is that this latter loop lifts to a loop in $\SU(r)$, and hence this loop is homotopically trivial since $\SU(r)$ is simply-connected. This proves Claim 2.
			
			\medskip

			By Claim 2, any gauge transformation induces an element $\eta'({\afu}) \in H^1(X, \bb{Z}_r)$. In a moment, we will show that this is exactly the parity $\eta({\afu})$. Putting this on hold for now, we prove the equivalence of the three conditions of the proposition. 
			
			Note that the proof of Claim 2 shows that $\eta'({\afu})$ vanishes if and only if ${\afu}_*\pi_1(P)$ is zero in $\pi_1(G)$. By general theory for the covering space $\SU(r) \rightarrow \PU(r)$, this happens if and only if ${\afu}$ lifts to a map $\widetilde{{\afu}} : P \rightarrow \SU(r)$. This proves the equivalence of the first two conditions.
			
			Now we prove the equivalence of the first and third conditions. Suppose $\eta'({\afu}) = 0$. View gauge transformations as sections of $P \times_G G \rightarrow X$. Note that $P \times_G \SU(r)$ is a cover of $P \times_G G$ (the former bundle is constructed using the representation of $\PU(r)$ on $\SU(r)$ induced by the adjoint representation of $\SU(r)$ on itself). Then the argument of the previous paragraph shows that ${\afu}$ lifts to a section $\widetilde{{\afu}}$ of $P \times_G \SU(r) \rightarrow X$. The fibers of this bundle are simply-connected, so this bundle is trivial over the 1-skeleton, $X^1$, of $X$ (in fact it is trivial over the 2-skeleton as well). This implies that ${\afu}$ can be homotoped to the identity over $X^1$. 
			
			Conversely, suppose ${\afu} \vert_{X^1}: X^1 \longrightarrow P \times_G G \vert_{X^1}$ can be homotoped to the identity. We have $\pi_2(G)$ is trivial, so ${\afu}$ can be homotoped to the identity over the 2-skeleton $X^2$ as well. Let $\iota: X^2 \hookrightarrow X$ denote the inclusion, which induces an isomorphism on $H^1$. The class $\eta'$ commutes with pullback, so we have 
			
			$$0  =  \eta'(\iota^* {\afu}) = \iota^*\eta'({\afu})$$
			which implies $\eta'({\afu}) = 0$. 
			
			\medskip
			
			Finally, we need to show that the classes $\eta({\afu}) = \eta'({\afu})$ agree. Fix a loop $\gamma: S^1 \rightarrow X$, and consider the map $\Gamma: S^1 \times S^1 \longrightarrow S^1 \times X$, given by $\Gamma(s,t)  = (s, \gamma(t))$. Then by the definition of $\eta$ and the functoriality of $t_2$, we have 
			
			$$\eta({\afu})\left[ \gamma \right] = t_2(\Gamma^*P_{\afu}) \in H^2(S^1 \times S^1 , \bb{Z}_r ) = \bb{Z}_r.$$
			As with the first Chern class on (connected) surfaces, this value can be computed by trivializing over a neighborhood of a point and its complement, and then measuring the homotopy class of the transition function (which is a map $S^1 \rightarrow G$). Since we are dealing with a torus, this can be done using a non-trivial circle, rather than a point as follows: Let $U$ be a tubular neighborhood of the circle $\left\{pt\right\} \times S^1 \subset S^1 \times S^1$, and $V$ the complement of this circle (so $U$ and $V$ are both cylinders). Then $\Gamma^*P_{\afu}$ can be trivialized over $U$ and $V$, and we can choose these trivializations so that the transition function on \emph{the second component} of the intersection $U \cap V \simeq S^1 \sqcup S^1$ is the identity. The class $\eta({\afu})\left[ \gamma \right] $ is then the homotopy class of the transition function over the first component, which is precisely the value $\eta'({\afu})\left[ \gamma\right] \in \bb{Z}_r$.

\end{proof}

			We end this section with a gauge-theoretic proof that the degree is a group homomorphism. Our proof is based on connections on the principal bundles; we refer the reader to Appendix \ref{AReviewOfConnections} for a review of connections.

		\begin{proposition}\label{homs}
		The degree is a group homomorphisms
		
		$$\deg: \pi_0 \: \G(P) \longrightarrow H^3(X, \bb{Z}).$$
		Moreover, when $X$ is a closed, connected and oriented 3-manifold, the degree satisfies

		\begin{equation}
		\deg(\afu) = \fracd{r}{4\pi^2 \kappa_r}\intd{\left[0, 1\right]} \left( \intd{X} \langle F_{\afa(s)} \wedge \partial_s \afa(s) \rangle \right),
		\label{constantthings}
		\end{equation}
		where $\afa: \left[0, 1\right] \rightarrow \A(P)$ is any path of connections from any fixed reference connection $\afa_0$ to $\afu^* \afa_0$. 
		\end{proposition}

		\begin{proof}		
		
		First consider a closed, connected, oriented 4-manifold $Z$ equipped with a principal $\PU(r)$-bundle $Q \rightarrow Z$. It follows from the definition of $q_4$, together with the Chern-Weil formula in Example \ref{ex4} that $q_4(Q)$ can be computed using the formula

		\begin{equation}
		q_4(Q)\left[Z\right] = \frac{r}{4\pi^2 \kappa_r} \intd{Z} \langle F_{\afA} \wedge 	F_{\afA} \rangle \in \bb{Z}.
		\label{chernweil}
		\end{equation}
		Here, $\kappa_r$ is as in Section \ref{PropertiesOfPSU(r)}, ${\afA}$ is any connection on $Q \rightarrow Z$, and $F_A$ is the curvature of the connection. 
		
		Now let $P \rightarrow X$ be a $\PU(r)$-bundle over a closed, connected, oriented 3-manifold $X$, and let $\afu$ be a gauge transformation on $P$. We will apply the formula (\ref{chernweil}) to $Z = S^1 \times X$ and $Q = P_{\afu}$, the mapping torus of $\afu: P \rightarrow P$. To do this, we need a connection on $P_{\afu}$, which we can construct as follows: Let $\A(P)$ denote the (affine) space of connections on $P$. Fix any connection ${\afa}_0 \in \A(P)$, and let ${\afa}: \left[0, 1\right]  \rightarrow \A(P)$ be any path from $\afa_0$ to $\afu^* \afa_0$. This defines a connection ${\afA}$ on $\left[0, 1\right] \times P \rightarrow \left[0, 1\right] \times X$ by declaring ${\afA} \vert_{\left\{s \right\} \times X } = \afa(s)$ for $s \in \left[0, 1\right]$. Moreover, ${\afA}$ descends to a connection on $P_{\afu}$, so by (\ref{chernweil}) we have
		
		$$\begin{array}{rcl}
		\deg(\afu) &=& \fracd{r}{8\pi^2 \kappa_r}\intd{\left[0, 1 \right] \times X} \langle F_{\afA} \wedge F_{\afA} \rangle. \\
		\end{array}$$
		The curvature decomposes into components as $F_{\afA} = F_{\afa(s)}  + ds \wedge \partial_s \afa(s)$, and so this can equivalently be written
		
		$$\deg(\afu) = \fracd{r}{4\pi^2 \kappa_r}\intd{\left[0, 1\right]} \left( \intd{X} \langle F_{\afa(s)} \wedge \partial_s \afa(s) \rangle \right),$$
		which is (\ref{constantthings}). (Our orientation convention is that $ds \wedge d\mathrm{vol}_X$ is a positive volume form on $I \times X$.) 
		
		That the degree is a group homomorphism follows from (\ref{constantthings}) and the additivity of the integral. Indeed, let $u, v \in \G(Q)$. Fix a path $\afa$ from $\afa_0$ to $u^*\afa_0$, and a second path $b$ from $u^* \afa_0$ to $v^*u^* \afa_0$. These can be concatenated to form a path $c$ from $\afa_0$ to $v^*u^* \afa_0$, and this gives
		
		$$\begin{array}{rcl}
		\deg(\afu \afv) & = & \fracd{r}{4\pi^2 \kappa_r}\intd{\left[0, 2 \right]} \left( \intd{X} \langle F_{c(s)} \wedge \partial_s c(s) \rangle \right)\\
		&&\\
		& = & \fracd{r}{4\pi^2 \kappa_r}\intd{\left[0, 1 \right]} \left( \intd{X} \langle F_{\afa(s)} \wedge \partial_s \afa(s) \rangle \right)\\
		&& \indent  +\fracd{r}{4\pi^2 \kappa_r}\intd{\left[1, 2 \right]} \left( \intd{X} \langle F_{b(s)} \wedge \partial_s b(s) \rangle \right)\\
		&&\\
		& = & \deg(\afu) + \deg(\afv).
		\end{array}$$
		\end{proof}

		\subsection{Application 1: Free actions on flat connections}

	 We suppose $X$ is a connected and oriented manifold of dimension 2 or 3, equipped with a principal $\PU(r)$-bundle $P \rightarrow X$. In case $X$ is 2-dimensional, we assume that $t_2(P)\left[X\right] \in \bb{Z}_r$ is a generator. If $X$ is 3-dimensional then we assume there is an embedding $\Sigma \hookrightarrow X$ of a connected, oriented surface such that $t_2(P) \left[ \Sigma \right] \in \bb{Z}_r$ is a generator. The next lemma shows that the subgroup $\ker \eta \subset \G(P)$ acts freely on the space of flat connections.

		\begin{lemma}\label{gaugestab0}
			Suppose $P \rightarrow X$ is as above and suppose ${\afa}$ is any flat connection on $P$. Then the stabilizer of ${\afa}$ in $\ker \eta \subseteq \G(P)$ is trivial:
			
			$$\left\{\left. {\afu} \in \G(P) \: \right|\: \eta({\afu}) = 0,\indent  {\afu}^*{\afa} = {\afa} \:  \right\} = \left\{e \right\}.$$
			\end{lemma}

	\begin{proof}
	The proof of this lemma follows just as in the proof of \cite[Lemma 2.5]{DS2}. We supply a sketch for convenience under the assumption that $X$ has dimension 3, and so there is some surface $\Sigma \subset X$ with $t_2(P \vert_{\Sigma}) \in \bb{Z}_r$ a generator. The case when $X$ is a surface is similar and a little easier. 
	
	The key point is that any ${\afu} \in \ker \eta$ lifts to an equivariant map $\widetilde{{\afu}}: P \rightarrow \SU(r)$. Fix a basepoint $p_0 \in P\vert_{\Sigma}$ and suppose ${\afu} \in \ker \eta$ fixes a flat connection ${\afa}$. We want to show that ${\afu}$ is the identity. Since ${\afu}$ fixes ${\afa}$, it follows that $\widetilde{{\afu}}(p_0)$ commutes with the $\SU(r)$-holonomy group $H_{\afa}(p_0) \subseteq \SU(r)$. In particular, $\widetilde{{\afu}}(p_0)$ commutes with the subgroup of $H_{\afa}(p_0)$ coming from the holonomy around loops lying entirely in the restriction $P\vert_{\Sigma}$. Since ${\afa}$ restricts to a flat connection on $P\vert_{\Sigma}$, and $t_2(P\vert_\Sigma) \in \bb{Z}_r$ is a generator, this subgroup is non-abelian \cite[p. 20]{DS2}. This implies that $\widetilde{{\afu}}(p_0)$ is central in $\SU(r)$, and so descends to the identity in $\PU(r)$. This argument holds for any $p_0 \in P\vert_{\Sigma}$ and so ${\afu}: P \rightarrow \PU(r)$ restricts to the identity map on $P \vert_{\Sigma} \subset P$. This argument further holds if $\Sigma$ is replaced by any closed, oriented surface $\Sigma' \subset X$ that is homologous in $X$ to $\Sigma$. Since any point in $P$ is contained in such a surface, it follows that ${\afu}$ is the identity on all of $P$, as desired. 
	\end{proof}

\subsection{Application 2: Existence of degree $d$ gauge transformations}

Fix a closed, connected, oriented 3-manifold $X$, and suppose we are handed a preferred closed, connected, oriented surface $\Sigma \subset X$. Let $P \rightarrow X$ be a principal $\PU(r)$-bundle, and assume that $t_2(P)$ is the reduction of an integral class. This implies that $P = \overline{P} \times_{\U(r)} \PU(r)$ is induced from a principal $\U(r)$-bundle $\overline{P} \rightarrow X$. Then we set

$$d \defeq c_1(\overline{P}) \left[ \Sigma\right] \in \bb{Z}.$$
(So $P$ satisfies the conditions of Application 1 when $d$ and $r$ are relatively prime.) 
	
	\begin{proposition}\label{degprop}
			Let $d \in \bb{Z}$, $P \rightarrow X$ and $\Sigma \subset X$ be as above. 

			\medskip			
			
			(a) Then there exists a gauge transformation $\afu \in \G(P)$ of degree $d$. Moreover, the parity $\eta(\afu) :H_1(X) \rightarrow \bb{Z}_r$ is given by the intersection number of a loop with $\Sigma$, reduced modulo $r$.
			
			\medskip
			
			(b) Suppose $d$ and $r$ are relatively prime. Then there exists a gauge transformation $\afu \in \G(P)$ of degree 1. Moreover, writing $md + nr =1$, the parity $\eta(\afu) :H_1(X) \rightarrow \bb{Z}_r$ is given by $m$ times the intersection number of a loop with $\Sigma$, reduced modulo $r$.
	\end{proposition}

In the case $r = 2$, Dostoglou and Salamon \cite[Lemma 2.3, Lemma A.2]{DS2} prove this statement by explicitly constructing the desired gauge transformation. We present a proof based upon the characteristic classes $t_2, q_4$.
		
		\begin{proof}[Proof of Proposition \ref{degprop}]
			Consider the cohomology class $t \in H^2(S^1 \times X, \bb{Z}_r)$ given by the mod $r$ reduction of 
			
			$$c_1(\overline{P}) +  ds \smallsmile \textrm{PD}_X \left[\Sigma \right] \in H^2(S^1 \times X)$$
			where $ds$ is the generator of $H^1(S^1)$ and $\textrm{PD}_X:H_2(X) \rightarrow H^1(X)$ denotes the Poincar\'{e} duality operator on $X$. Let $C$ be the Pontryagin square appearing in (\ref{pontrelation}). Then since $t$ is the reduction of an integral class, we can compute as follows
			
			$$\begin{array}{rcl}
			Ct & \equiv_{2r} & \left\{ \begin{array}{ll}
								2  c_1(\overline{P}) \smallsmile ds \smallsmile \textrm{PD}_X\left[\Sigma \right] &  \textrm{if $r$ is even}\\
								4  c_1(\overline{P}) \smallsmile ds \smallsmile \textrm{PD}_X\left[\Sigma \right] &  \textrm{if $r$ is odd}
								\end{array}\right.\\
								&&\\
								& \equiv_{2r} & \left\{ \begin{array}{ll}
										2d \cdot  e &  \textrm{if $r$ is even}\\
										4d \cdot  e & \textrm{if $r$ is odd}
										\end{array}\right.
								\end{array}$$
		where $e \in H^4(S^1 \times X)$ is the positive generator. This gives
		
		$$\left. \begin{array}{cr}
											\textrm{$r$ even} &		(r+1)Ct\\
											\textrm{$r$ odd} & 		\fracd{r+1}{2} C t
													\end{array}\right\}								\;		 \equiv_{2r}  \;  2d(r+1) \cdot e  \;\equiv_{2r} \;  2d \cdot e.$$
			This is exactly the relation (\ref{pontrelation}) with $t$ replacing $t_2$ and $2d \cdot e$ replacing $q_4$. Since we are on a 4-manifold, it follows from L.M. Woodward's classification that there is a principal $\PU(r)$-bundle $Q \rightarrow S^1 \times X$ with
			
			$$t_2(Q)  = t, \indent q_4(Q) = 2d \cdot e.$$
			This bundle $Q$ restricts to the bundle $P$ on each fiber of $\left\{pt \right\} \times X$. In particular, by Proposition \ref{donfloer}, this bundle is of the form $Q = P_{\afu}$ for some gauge transformation ${\afu} \in \G(P)$. It follows from our definition of the degree that ${\afu}$ is a gauge transformation of degree $d$ with $\eta(\afu)$ given by the mod $r$ reduction of $\mathrm{PD}_X\left[ \Sigma \right]$. This proves the first statement.
			
			Now suppose $d$ and $r$ are relatively prime. Then there is some integer $m$ with
			
			$$2md \equiv_{2r} 2.$$
			The same type of argument given above shows that there is a gauge transformation ${\afu} \in \G(P)$ with 
			
			$$t_2(P_{\afu}) = m \cdot t , \indent q_4(P_{\afu}) = 2.$$
			This shows $\deg(\afu) = 1$ and $\eta(\afu)$ is the mod $r$ reduction of $m\cdot \mathrm{PD}_X\left[ \Sigma \right]$.
			\end{proof}

		\subsection{Application 3: Circle fibrations and the group $\G_{\Sigma}$}
		
		Suppose $X$ is closed, connected, oriented 3-manifold equipped with a smooth function $f: X \rightarrow S^1$. Assume $f$ is not homotopically trivial, and that each fiber is connected. 
		
		\begin{example}\label{ex5}
				Take $X = S^1 \times \Sigma$ for a closed, connected, oriented surface $\Sigma$, and let $f$ be the projection to the first factor.  
		\end{example} 

		Fix a map $\gamma: S^1 \rightarrow X$ that is a section of $f$ 
		
		$$\mathrm{Id}_{S^1} = f \circ \gamma.$$ 
		Then for each $r \geq 2$ and $d \in \bb{Z}_r$, there is a unique principal $\PU(r)$-bundle $P \rightarrow X$ such that $t_2(P) \in H^2(X, \bb{Z}_r)$ is Poincar\'{e} dual to the class $d \left[ \gamma \right] \in H^1(X, \bb{Z}_r)$. Since $\left[ \gamma \right]$ is the reduction of an integral class, it follows that $t_2(P)$ is as well. Moreover, if we set
		
		$$\Sigma \defeq f^{-1}(pt),$$
		where $pt \in S^1$ is a regular value, then we have
		
		$$t_2(P) \left[ \Sigma \right] = d.$$
		This follows because the intersection number of $\left[ \Sigma \right]$ with $\left[ \gamma \right]$ is 1.

		Consider the subgroup
		
		$$\G_{\Sigma} \defeq \left\{ \afu \in \G(P) \: \vert \: \eta_{\Sigma} \left( \afu\vert_{\Sigma}  \right)= 0 \right\},$$
		where $\eta_\Sigma$ is the parity operator for $P \vert_{\Sigma} \rightarrow \Sigma$. Then there is a sequence of inclusions
		
		$$\G_0(P) \subset \ker \eta \subset \G_{\Sigma} \subset \G(P).$$
		The first and last inclusions are always strict. It follows from Proposition \ref{degprop} (a) that the middle inclusion is strict when $d \in \bb{Z}_r$ is not zero.

		\begin{proposition}
			Suppose $d \in \bb{Z}_r$ is a generator. Then there is a canonical isomorphism $\G_\Sigma / \G_0(P) \cong \bb{Z}$. Moreover, the generator of $\G_\Sigma / \G_0(P) $ is the homotopy class of the degree 1 gauge transformation from Proposition \ref{degprop} (b).
		\end{proposition}

		\begin{proof}
		Let $\afu_1 \in \G(P)$ be a degree 1 gauge transformation as in Proposition \ref{degprop} (b). Then the parity $\eta(\afu_1)$ measures the intersection number of a loop with $\Sigma$. This implies $\afu_1 \in \G_{\Sigma}$, since any loop in $\Sigma$ can be displaced (in $X$) from $\Sigma$. Since the degree and parity of $\afu_1$ are specified, it follows that the homotopy class $\left[ \afu_1 \right] \in \G(P) / \G_0(P)$ of $\afu_1$ is uniquely determined. To prove the proposition, it suffices to show that every gauge transformation $\afu \in \G_\Sigma$ is a power of $\afu_1$, up to homotopy. That is, we need to show
		
		$$\eta(\afu) = k \cdot \eta(\afu_0), \indent \deg(\afu) = k \cdot \deg(\afu_0)$$
		for some $k \in \bb{Z}$. 
		
		Since $\afu_0$ has degree 1, we obviously have
		
		\begin{equation}\label{specialk}
		\deg(\afu) = k\cdot  \deg(\afu_0)
		\end{equation}
		for some $k \in \bb{Z}$. As for the parity, note that Proposition \ref{degprop} (b) implies
		
			$$\eta(\afu_0) \left[ \gamma \right] = dm \indent \mathrm{mod}\; r,$$		
		where $m \in \bb{Z}$ is chosen so that $md + nr = 1$ for some $n  \in \bb{Z}$. (Here and below we freely identify elements of $\bb{Z}_r$ with any lift in $\bb{Z}$; our statements are independent of the choice of lift.) Note that the product $md$ is relatively prime to $r$ and so descends to a generator of $\bb{Z}_r$. In particular, there is some $l \in \bb{Z}$ such that
		
		$$\eta(\afu) \left[ \gamma \right] = l md \cdot \eta(\afu_0) \left[ \gamma \right].$$
		The first homology $H_1(X)$ is generated by $H_1(\Sigma)$ and $\left[ \gamma \right]$, subject to certain relations. By definition, the parity of each element of $\G_\Sigma$ vanishes on $H_1(\Sigma)$ and so we have 
		
		\begin{equation}\label{speciallmd}
		\eta(\afu)  = l md\cdot \eta(\afu_0) \in H^1(X, \bb{Z}_r).
		\end{equation}
		We will be done if we can show that $lmd$ is the mod-$r$ reduction of $k$ from (\ref{specialk}). This follows from (\ref{pontrelation}): Since $t_2(P)$ is the reduction of an integral class, we can use (\ref{computeC}) to compute (\ref{pontrelation}). This combines with the definition of the degree and parity in terms of $t_2$ and $q_4$ to give
		
		\begin{equation}\label{this}
		ds \smallsmile \deg(\afv) =  t_2(P) \smallsmile ds \smallsmile \eta(\afv) \indent \mathrm{mod}\: r
		\end{equation}
		for any gauge transformation $\afv \in \G(P)$. This is an equation in $H^4(S^1 \times X,\bb{Z}_r)$. Apply (\ref{this}) with $\afv = \afu$, and then use (\ref{specialk}) and (\ref{speciallmd}) to get
		
		$$ k \cdot  ds \smallsmile \deg(\afu_0) = lmd \cdot  t_2(P) \smallsmile ds \smallsmile \eta(\afu_0) \indent \mathrm{mod}\: r.$$
		Since $\afu_0$ has degree 1, this gives $k = lmd\; \mathrm{mod}\: r$, as desired.

		\end{proof}

\appendix

\section{A review of connections}\label{AReviewOfConnections}

	This appendix gives a fast review of connections. Since connections are only used in the proof of Proposition \ref{homs}, we only supply the details that we will need for the proof. For a more comprehensive treatment, we refer the reader to $\cite{KN}$, or to \cite[Section 4]{DunIntro} for notation similar to the notation used here. 
	
	Let $V \rightarrow X$ be a vector bundle. We assume for concreteness that $V$ is a real vector bundle, but the definitions we give below carry over to complex vector bundles with only minor modifications. Given $k \in \bb{N}$, the vector bundle $V$ can be tensored with $\Lambda^k T^*X$ to form a new vector bundle
		
		$$\left(\Lambda^k T^*X\right) \otimes V \longrightarrow X.$$
We denote sections of this bundle by $\Omega^k(X, V)$, and these section should be viewed as `$k$-forms on $X$ with values in $V$'. 

A {\bfseries connection} on $V$ is a linear map $d_{\afA}: \Omega^0(X, V) \longrightarrow \Omega^1(X, V)$ that satisfies the following Leibniz rule
	
	$$d_{\afA} (f \mu ) = df \otimes \mu + f d_{\afA} \mu$$
	for $f \in C^\infty(X)$ and $\mu \in \Omega^0(X, P(\frak{g}))$. For each $k \in \bb{N}$, any connection $d_{\afA}$ has a canonical extension to a linear map 
		
		$$d_{\afA}: \Omega^k(X, V) \longrightarrow \Omega^{k+1}(X, V)$$
		satisfying the obvious Leibniz rule. The composition $d_{\afA} \circ d_{\afA}$ is linear over $C^\infty(M)$ and so acts via a degree 2 algebraic operator ${\cal{F}}_{\afA, V} $ called the {\bfseries curvature}. That is, ${\cal{F}}_{\afA, V} \in \Omega^2(X, \mathrm{End}(V))$ and 
		
		$$d_{\afA } \circ d_{\afA} \mu = {\cal{F}}_{\afA, V} (\mu) $$
		for $\mu \in \Omega^0(X, V)$. Here $\mathrm{End}(V) = V^* \otimes V \longrightarrow X$ is the bundle of endomorphisms of $V$. It follows immediately that any connection $\afA$ satisfies the {\bfseries Bianchi identity}
		
		$$\left[ d_{\afA}, {\cal{F}}_{\afA, V} \right] \defeq d_{\afA} \circ {\cal{F}}_{\afA, V}  - {\cal{F}}_{\afA, V} \circ d_{\afA}  = 0.$$

		We denote the space of connections on $V$ by $\A(V)$. It follows that $\A(V)$ is an affine space modeled on the vector space $\Omega^1(X, V)$. As a consequence, the set $\A(V)$ is contractible and the tangent space at $\afA \in \A(V)$ can be canonically identified with $\Omega^1(X, V)$.

		\medskip
		
		Now let $G$ be a Lie group and $\frak{g}$ its Lie algebra. Assume that $\frak{g}$ is semi-simple. Fix a principal $G$-bundle $P \rightarrow X$. Consider the adjoint representation $\mathrm{Ad}: G \rightarrow \GL(\frak{g})$ and use this to form the associated vector bundle $P(\frak{g})\rightarrow X$ as in Example \ref{ex0}. We will say that a {\bfseries connection} on $P$ is a connection on the vector bundle $P(\frak{g}) \rightarrow X$. We denote the space of connection on $P$ by $\A(P)$; that is, $\A(P) = \A(P(\frak{g}))$. 
		
		\begin{remark}
		The definition of a connection on a principal bundle that we give here is not the typical definition given in the literature (i.e., as $\frak{g}$-valued 1-form on $P$ satisfying certain properties). However, since we have assumed the Lie algebra is semi-simple, the adjoint representation is faithful and so there is no information loss in dealing with the bundle $P(\frak{g})$. That is, a connection on $P$ as we have defined it here induces a unique connection in the standard sense, and vice-versa.
		\end{remark}

		It turns out that the curvature ${\cal{F}}_{\afA, P(\frak{g})}$ satisfies

		\begin{equation}\label{un}
		{\cal{F}}_{\afA, P(\frak{g})}\left(\mu \right) = \mathrm{ad}(F_{\afA}) \mu = \left[ F_{\afA}, \mu \right] \indent \indent \forall \mu \in \Omega^0(X, P(\frak{g})),
		\end{equation}
		for some $F_{\afA} \in \Omega^2(X, P(\frak{g}))$. The form $F_{\afA} $ will be called the {\bfseries curvature form} of $\afA$. Since $\frak{g}$ is semi-simple,  $F_{\afA}$ is the unique 2-form for which (\ref{un}) holds. In terms of the curvature form, the Bianchi identity takes the form $d_{\afA} F_{\afA} = 0$, where the concatenation is the given action of $d_{\afA}$ on the 2-form $F_{\afA}$.

		The gauge group $\G(P)$ acts naturally on $\A(P)$. We denote the action of $\afu\in \G(P)$ on $\afA \in \A(P)$ by $\afu^* \afA$, where $\afu^*\afA$ is the connection defined by the formula
		
		$$d_{\afu^*\afA} \mu \defeq \afu^{-1} d_{\afA} \left(\afu\mu \right),$$
		for $\mu \in \Omega^0(X, P(\frak{g}))$ (on the right we are viewing gauge transformations as sections of $P \times_G G \rightarrow X$, and we have chosen a faithful matrix representation of $G$ so that it makes sense to multiply elements of $G$ with elements of $\frak{g}$). A computation shows that the curvature form is equivariant $F_{\afu^*\afA} = \mathrm{Ad}(\afu^{-1}) F_{\afA}$.

Now suppose $\frak{g}$ is equipped with an $\mathrm{Ad}$-invariant inner product $\langle \cdot, \cdot \rangle$. Then this inner product determines a metric on the bundle $P(\frak{g})$, which we denote by the same symbol. This metric combines with the wedge to form a bilinear map 
		
		$$\Omega^j(X, P(\frak{g})) \otimes \Omega^{k}(X, P(\frak{g})) \longrightarrow \Omega^{j+k}(X), \indent \mu\otimes \nu  \longmapsto  \langle \mu \wedge \nu \rangle$$
		with values in the (usual) space of $\bb{R}$-valued forms on $X$.

		\begin{example}\label{ex3}
		Suppose we are given a path of connections 
		
		$$\bb{R}  \longrightarrow  \A(P), \indent s \longmapsto  \afA(s).$$ 
		Then for each $s \in \bb{R}$, the derivative $\partial_s \afA(s) \in \Omega^1(X, P(\frak{g}))$ is a $P(\frak{g})$-valued 1-form, and so 
		
		$$s \longmapsto \langle F_{\afA(s)} \wedge \partial_s \afA(s) \rangle \in \Omega^3(X)$$
		is a path of 3-forms on $X$.
		\end{example}

		\begin{example}\label{ex4}
			Suppose $X$ is a closed, connected, oriented $4$-manifold and fix a principal $\PU(r)$-bundle $P \rightarrow X$. Then we can form the complex vector bundle $V \defeq P(\frak{pu}(r)) \otimes \bb{C} \rightarrow X$. Let $\afA$ be a connection on $P(\frak{pu}(r))$. This has a canonical extension to the complexification $V$. 
			
			Let $ch(V)$ denote the Chern character, and $\left[ch(V)\right]$ its cohomology class. Since $X$ has dimension 4, we have

		$$\left[ch(V)\right] = \mathrm{rank}(V) + c_1(V) + \left(\fracd{1}{2}c_1(V)^2 - c_2(V)\right).$$		
		where $c_1(V)$ and $c_2(V)$ are the first and second Chern classes of $V$. These can be computed using the {\bfseries Chern character formula}
		
		$$ch(V) = \mathrm{Tr}\left( \exp\left(\fracd{i {\cal{F}}_{\afA}}{2 \pi} \right) \right).$$ %You use a connection on the vector bundle, and then the curvature is the one induced on the associated frame bundle. This is something with values in Ad(V), and so is \left[ F_A , \cdot \right]$. 
		The right-hand side is a formal power series in the algebra of differential forms on $X$, and the trace is the usual one on $\frak{pu}(r) = \frak{su}(r)$. Since $X$ has dimension four, this power series truncates after degree four, and so is a well-defined element of $\Omega^\bullet(X)$. 
			
			 Since $V$ is the complexification of a real vector bundle, it follows that $c_1(V) = 0$. To see this, note that, in general, we have $c_1(\overline{V}) = - c_1(V)$, where $\overline{V}$ is the conjugate bundle to $V$. That $V$ is a complexification gives $\overline{V} \cong V$ and so $c_1(V) = 0$. In particular, all of the interesting information of the Chern character $\left[ch(V)\right]$ is contained entirely in the degree 4 term. Using the Chern character formula, we arrive at the following {\bfseries Chern-Weil formula}:

		\begin{equation}\label{cw2}
		c_2(V)\left[X\right]  = \fracd{1}{8\pi^2} \intdd{X} \: \mathrm{Tr}\left( {\cal{F}}_{\afA} \wedge {\cal{F}}_{\afA} \right) = \fracd{r}{4\pi^2} \intdd{X} \: \mathrm{tr}\left( {F}_{\afA} \wedge {{F}}_{\afA} \right) ,
		\end{equation}
where we have used (\ref{traces}). Note that this is always an integer. Using the Bianchi identity, one can check directly that this is independent of the choice of connection $\afA$.

		\end{example}

\end{document}